\documentclass[10pt,leqno, english]{amsart}
\title[Twisted Geometric  K-homology ]
{Twisted  Geometric  K-homology  for  proper  actions  of  discrete groups. }

\author{No\'e B\'arcenas  }
        \address{Centro de Ciencias Matem\'aticas. UNAM \\ Ap.Postal 61-3 Xangari. Morelia, Michoac\'an M\'EXICO 58089}
         \email{barcenas@matmor.unam.mx}
         \urladdr{http://www.matmor.unam.mx/~barcenas}

         \date{\today}

\usepackage{hyperref}
\usepackage{pdfsync}
\usepackage{calc}
\usepackage{enumerate,amssymb}
\usepackage[arrow,curve,matrix,tips,2cell]{xy}
  \SelectTips{eu}{10} \UseTips
  \UseAllTwocells

\usepackage{verbatim}

\DeclareMathAlphabet\EuR{U}{eur}{m}{n}
\SetMathAlphabet\EuR{bold}{U}{eur}{b}{n}

\makeindex             

\theoremstyle{plain}
\newtheorem{theorem}{Theorem}[section]
\newtheorem{lemma}[theorem]{Lemma}
\newtheorem{proposition}[theorem]{Proposition}
\newtheorem{corollary}[theorem]{Corollary}

\theoremstyle{definition}
\newtheorem{definition}[theorem]{Definition}

\newtheorem{condition}[theorem]{Condition}

{\catcode`@=11\global\let\c@equation=\c@theorem}

\newcommand{\comsquare}[8]                   
{\begin{CD}
#1 @>#2>> #3\\
@V{#4}VV @V{#5}VV\\
#6 @>#7>> #8
\end{CD}
}

\newcommand{\xycomsquare}[8]                   
{\xymatrix
{#1 \ar[r]^{#2} \ar[d]^{#4} &
#3 \ar[d]^{#5}  \\
#6\ar[r]^{#7} &
#8
}
}

\newcommand{\calu}{{\mathcal U}}

\newcommand{\IC}{{\mathbb C}}

\newcommand{\IR}{{\mathbb R}}

\newcommand{\IZ}{{\mathbb Z}}

\newcommand{\ind}{\operatorname{ind}}

\newcommand{\Spin}{\operatorname{Spin}}

\newcommand{\SO}{\operatorname{SO}}

\newcommand{\pt}{\{\bullet\}}

\newcommand{\cover}{\mathcal{N}_{G} \mathcal{U}}

\newcommand{\ktheory}[3]{K_{#1}^{#2}(#3, P)}
\newcommand{\Fred}{\ensuremath{{\mathrm{Fred}}}}
\newcommand{\kgeom}[3] {K^{\rm geo}_ {#1}({#2},{#3})}

\newcommand{\UU}{\mathcal{U}}
\newcommand{\HH}{\mathcal{H}}

\newcommand{\KK}{\mathcal{K}}

\newcommand{\eub}[1]{\underline{E}#1}              

\newcommand{\higherlim}[3]{{\setbox1=\hbox{\rm lim}
        \setbox2=\hbox to \wd1{\leftarrowfill} \ht2=0pt \dp2=-1pt
        \mathop{\vtop{\baselineskip=5pt\box1\box2}}
        _{#1}}^{#2}#3}

\newcommand{\version}[1]                       
{\begin{center} last edited on #1\\
last compiled on \today\\
name of texfile: \jobname
\end{center}
}

\newcounter{commentcounter}

\begin{document}

\begin{abstract}
We  define  Twisted  Equivariant  $K$-Homology  groups   using  geometric  cycles.  We  compare  them  with  analytical  approaches  using  Kasparov KK-Theory  and  (twisted) $C^*$-algebras  of  groups  and  groupoids. 

\end{abstract}

  \maketitle

\typeout{----------------------------  linluesau.tex  ----------------------------}


\setcounter{section}{0}

\section{Introduction}

In  this  work,   we introduce  a  geometric  notion  of  twisted  equivariant  $K$-Homology  groups  for  proper  $G$-CW  complexes.

  Computational   evidence  \cite{barcenasvelasquez}, \cite{echterhofflueckphillipswalters},  as  well  as analytical  methods  related  to the  Dirac-Dual  Dirac  method  for  the  Baum-Connes  conjecture \cite{echterhoffchabertgroupextensions}, \cite{echterhoffemersonkim} have  suggested the relation  of  a   Twisted  Equivariant  K-Homology  Theory  to the  topological  $K$-Theory  of  twisted  group  $C^*$-algebras.
  
However,  the  definition  of  such  an  invariant  has  been  only given  in particular  cases    \cite{echterhoffchabertgroupextensions}, \cite{echterhoffktheorytwisted} using  Kasparov KK-Theory and  rather  ad-hoc methods inspired  in  ideas related the  solution  of  the  Connes-Kasparov conjecture  by   Chabert, Echterhoff and  Nest \cite{echterhoffchabertnest}.  

 We   describe  twisted  equivariant  $K$-homology  groups using  geometric  cycles  for  proper  actions  of  discrete  groups  on $G$-CW  complexes. We extend  and  generalize  results  from  \cite{baumhigsonschickproper}, \cite{baumcareywang}  to  the  twisted,  respectively  equivariant  case.   
 
The  main  methods   in this  note are  a  combination of  consequences  of  the  pushforward  homomorphism  constructed  in \cite{barcenascarrillovelasquez}, and   ideas  ultimately  originated  in  geometric  pictures  of  $K$-homology \cite{baumhigsonschickproper}, \cite{baumcareywang}.

 We compare  the cycle  version  using  an  index  map  to  an  equivariant  Kasparov  $KK$-Theory  group. We  also  introduce  some  computational   tools,  notably  a  spectral  sequence  to  compute  the  equivariant  $K$-homology  groups  described  here and we compare  the  twisted  equivariant  $K$-homology  invariants   of  the  classifying  space   for  proper  actions  to  the  topological  $K$-theory  of  twisted  versions  of  the  reduced  $C^*$  algebra of  a  discrete group.

This  work  is  organized  as  follows:  

In  section  \ref{sectionreview},  we  recall  definitions,  methods,  and  results  related  to  twisted  equivariant  $K$-Theory  for  proper  actions \cite{barcenasespinozajoachimuribe}, \cite{barcenasespinozauribevelasquez}, \cite{barcenascarrillovelasquez}. 

In  section  \ref{sectiongeometricapproach},  geometric  cycles   for  twisted  equivariant  $K$-homology  are  introduced,  as  well as  appropriate equivalence  relations (bordism,  isomorphism  and vector bundle  modification). It  is  proved  in  Theorem \ref{theoremhomologicproperties}, that the  geometric  Twisted  Equivariant $K$-Homology  groups  satisfy  dual  homological properties  to  the  twisted  equivariant  $K$-Theory  groups  in  \cite{barcenasespinozajoachimuribe}. In  that  section  computational  methods, including a  spectral  sequence abutting to  Twisted  Equivariant  $K$-Homology      are  addressed.
 
In  section  \ref{analyticalkhom}, the construction  is  compared  to  equivariant  Kasparov $KK$-groups  via  the  definition  of  an index  map,  which  is  proved in Theorem \ref{theoremcoincidencekhomology} to  give  an  isomorphism  between  both  constructions in the  case  of  a  proper,  finite $G$-CW complex. On  the  way,  several  results,  including  versions  of $KK$-Theoretical Poincar\'e  Duality \cite{echterhoffemersonkim}, \cite{echterhoffktheorytwisted} and  consequences  of  the  Thom  isomorphisms and  Pushforward  Maps  for  Twisted Equivariant $K$-Theory of  \cite{barcenascarrillovelasquez} are discussed.

Finally, the  relation  of  the   construction  depicted  below   with   the   $K$-Theory  of twisted  crossed  products \cite{echterhoffwilliams}, \cite{packerraeburn} is  established in  section \ref{Section Crossed  Products}.

\subsection{Acknowledgements}

The  author  thanks  the  support  of  DGAPA-UNAM, and  CONACYT  research   grants. 

Parts  of this  manuscript  were  written  during a  visit  to  Universit\'e Toulouse III Paul Sabatier, which  was partially  founded  by  the  Laboratoire  International Solomon Lefschetz.

The  author  benefited  from conversations  and  joint  work  with  Mario Vel\'asquez  and  Paulo Carrillo.

\tableofcontents

\section{Review  on Twisted  Equivariant  $K$-Theory and  Pushforward}\label{sectionreview}

Twisted  Equivariant  K-Theory  for  proper  actions  of  discrete  groups  was  introduced  in \cite{barcenasespinozajoachimuribe}. 
We  will  recall   definitions,  results and  methods   from \cite{barcenasespinozajoachimuribe} and  \cite{barcenascarrillovelasquez}.

Let $\HH$ be a separable Hilbert space and 
$$\UU(\HH):= \{ U : \HH \to \HH \mid U\circ U^*= U^*\circ U = \mbox{Id} \}$$ the group
of unitary operators  on $\HH$ with the  $*$-strong  topology.
 
 Consider   the    group $P\UU(\HH)$ with   the  topology determined  by  the exact  sequence
$$1\to S^1\to  \UU(\HH)\to P\UU(\HH)\to  1 .$$

Let $\mathcal{H}$ be a  Hilbert  space. A continuous homomorphism $a$ defined  on a  Lie  group $G$,  $a : G \to P\UU(\HH)$ is called  
stable if the unitary representation $\HH$ induced by the
homomorphism $\widetilde{a}: \widetilde{G}= a^*\UU(\HH) \to
\UU(\HH)$ contains each of the irreducible representations of
$\widetilde{G}$

\begin{definition}\label{def projective unitary G-equivariant stable bundle}
Let  $X$ be  a proper $G$-CW  complex.  A projective unitary $G$-equivariant stable bundle over $X$
is a principal $P\UU(\HH)$-bundle
$$P\UU(\HH) \to P \to X$$
where $P\UU(\HH)$ acts on the right, endowed with a left $G$-action lifting the action on $X$ such that:
\begin{itemize}
\item the left $G$-action commutes with the right $P\UU(\HH)$
action, and \item for all $x \in X$ there exists a
$G$-neighborhood $V$ of $x$ and a $G_x$-contractible slice $U$ of
$x$ with $V$ equivariantly homeomorphic to $ U \times_{G_x} G$
with the action $$G_x \times (U \times G) \to U \times G, \ \ \ \
k \cdot(u,g)= (ku, g k^{-1}),$$ together with a local
trivialization
$$P|_V \cong  (P\UU(\HH) \times U) \times_{G_x} G$$ where the action of the isotropy group
is:
 \begin{eqnarray*}
 G_x \times \left( (PU(\HH) \times U) \times G \right)& \to & (P\UU(\HH) \times U) \times G
 \\
\ \left(k , ((F,y),g)\right)& \mapsto & ((f_x(k)F, ky), g k^{-1})
\end{eqnarray*} with $f_x : G_x \to PU(\HH)$ a fixed stable
homomorphism.
\end{itemize}
\end{definition}

Projective  unitary  $G$-equivariant stable  bundles are  the  main  tool  to  introduce twisted  equivariant  $K$-Theory  as  sections  of  Fredholm  bundles. Other  topologies  have  been   used  on  the  relevant  operator  spaces \cite{barcenasespinozajoachimuribe}, \cite{barcenasvelasquezatiyahsegal}. However,  they  are  all  easily  seen  to  agree (\cite{barcenasvelasquezatiyahsegal}, Section 1.1. See  also  \cite{espinozauribeunitaryoperators}). 

Projective   unitary  $G$-equivariant   stable bundles  have  been  classified  in \cite{barcenasespinozajoachimuribe}, Theorem 4.8 in  page  1341.  The  isomorphism  classes  of   projective   unitary  stable  $G$-equivariant  bundles over  a  proper  $G$-CW  complex $X$  are  in  correspondence  with  the  third  Borel  Cohomology  Group  $H^3(X\times_G EG, \IZ)$. There  exists  an  intrinsic abelian  group  structure on the   isomorphism  classes  of  projective  unitary  stable  $G$-equivariant  bundles for which  this correspondence is  an isomorphism  of  abelian groups.

  Let  $X$  be a  proper  $G$-CW  complex. A $G$-Hilbert  bundle  over $X$  is  a locally  trivial  bundle $E\to X$  with  fiber  on  a Hilbert  space  $\HH$ and  structural  group  the group of  unitary operators  $\UU(\HH)$ with the $*$-strong operator  topology. 
  
The Bundle  of  Hilbert-Schmidt operators with  the  $*$-strong  topology  between  Hilbert  Bundles  $E$ and $F$  will  be  denoted by $L_{HS}(E,F)$.

Let  $P$  be  a projective  unitary  stable  $G$-equivariant  bundle. Consider the projective  unitary  stable  $G$-equivariant  bundle  $-P$  associated  to  the  additive  inverse   of  $P$   using  the  isomorphism  from  Theorem 4.8  of  \cite{barcenasespinozajoachimuribe}  to the Borel  cohomology  group and  form  similarly  the  class $P+ -P$. The  classification  of  Projective  Unitary  stable  $G$-equivariant  bundles \cite{barcenasespinozajoachimuribe}  gives:

\begin{proposition}
There  exists  a  Hilbert bundle $E$  over  X    such  that  the  zero  class  $P+-P $ admits  a  representative   given  by  the   principal  bundle  associated  to  the  Bundle  $L_{HS}(E_{P},{E_{P}}^*)$ of  Hilbert-Schmidt  homomorphisms  of  $E_{P}$ to  the  dual  bundle ${{E}_P}^*$. 
\end{proposition}

  \begin{definition}\label{definitiongeometricktheory}
Let $X$ be a connected $G$-space and $P$ a projective unitary stable
  $G$-equivariant bundle over $X$,  with associated  Hilbert  space $\HH$.

 Denote  by $\Fred(\HH)$  the  space  of  $*$-strong continuous,  Fredholm  operators  on $\HH$ and  recall  that  the group  of  projective  unitary  operators  $P\UU(\HH)$  in  the  ${}^*$-strong  topology  acts  continuously   by  conjugation on $\Fred( \HH)$. The  twisted  equivariant  $K$-Theory  groups    denoted  by   $K^{-p}_G(X;P)$   are  defined in  \cite{barcenasespinozajoachimuribe} as  the $p-th$  homotopy  groups (based a at  a  suitable  section $s$)  of  the  space  of   $G$-invariant sections  of  the  bundle  with  fibre  $\Fred(\HH)$  and  structural  group  $ P \UU( \HH)$  associated  to  $P$. In  symbols, 
      
$$   K^{-p}_G(X;P):= \pi_p \left( \Gamma(X;\Fred({P})_{s^*})^G, s \right).$$

Given  a  $G$-CW  pair  $(X,A)$,  one  defines  the  relative   twisted  equivariant $K$- Theory  groups     $K^{-p}_G(X, A;P)$  in an analogous  way.  
\end{definition}

Let  $G$ be  a  discrete  group. A   proper, $G$-compact manifold  is  a  smooth, orientable   manifold  with  an  orientation  preserving,  proper and  smooth   action  of  the  group  $G$,  in  such a  way  that  the  quotient  space  $G/M$  is  compact.

Given a  proper  $G$-compact  manifold  $M$ of  dimension  $n$,  consider  the  tangent  bundle  $T M$.  It  is  a  real $G$-vector  bundle over  $X$ of  rank  $n$.  Let $F(M)$ be  the  associated $\SO(n)$-principal,  $G$-equivariant  bundle.

 A  $G$-spin  structure on $M$  is  a  homotopy  class  of  a  reduction  of  the   bundle  $F(M)$ to a  $\Spin(c)$-principal, $G$-equivariant  bundle.  Given a   choice of  the  $\Spin(c)$-structure  for  $M$  we  denote  by  $-M$  the  manifold  with the  opposite  $\Spin(c)$-structure.

  This  is  the homotopy  class  of  the  reduction   of  the  bundle  $F(M)$ to  the $\Spin(c)$-principal,  $G$-equivariant  bundle, where  the   right action  of $\Spin(c)$  is  twisted  by  an  automorphism $\epsilon: \Spin(c)\to \Spin(c)$ induced  from the  orientation  reversing  automorphism $- {\rm  id} : \IR^n\to \IR^n$.

We  recall  one of  the  main  results  from  \cite{barcenascarrillovelasquez}, obtained  using  a   version  of   the  Thom  Isomorphism and  suitable deformation  index  maps  in  Twisted  Equivariant  $K$-theory. It  is  stated  there (Proposition 5.18) in  the  more  general  context  of  Metalinear  structures,  but, as  observed  there,  the  proof  readily  applies to the  $\Spin(c)$-case.

\begin{theorem} [Pushforward morphism for proper $G$-manifolds]\label{theopushforward}

Let $X,Y$ be two  proper, smooth  $\Spin(c)$- $G$- manifolds. Let  $P$ be  a  projective  unitary  stable  $G$-bundle  over  $Y$. Let  $f:X\longrightarrow Y$ be  a  map  such the  $G$-vector  bundle  $T X\oplus f^*TY$  is  an  equivariant $\Spin(c)$-vector bundle of even rank. Then,   there  exists   a  natural  pushforward  homomorphism

$$f!: K_{G}^0(X, f^*(P))\to K_{G}^0(Y, P).$$
\end{theorem}

\section{Geometric  Approach  to  twisted  equivariant K-Homology}\label{sectiongeometricapproach}

We  describe  now   a   geometric  version of  Twisted  Equivariant  $K$-Homology. The  results  in  this  section extend  and  generalize  methods   by  Baum-Higson-Schick  and  Baum-Carey-Wang \cite{baumhigsonschickproper}, \cite{wanggeometriccycles}, \cite{baumcareywang}.

    Recall  that  a  $G$-CW  complex structure  on  the  pair $(X,A)$  consists  of a  
  filtration of  the $G$-space $X=\cup_{-1\leq n } X_{n}$ with $X_{-1}=\emptyset$,  
  $X_{0}=A$ and every  space   is inductively  obtained  from  the  previous  one   
  by  attaching  cells  with  pushout  diagrams  
  $$\xymatrix{\coprod_{\lambda \in I_n} S^{n-1}\times G/H_{\lambda} \ar[r] \ar[d] & X_{n-1} \ar[d] \\ 
  \coprod_{\lambda \in I_n}D^{n}\times G/H_{\lambda} \ar[r]& X_{n}.} $$ 

The  set $I_n$  is  called  the  set  of  $n$-cells. A  $G$-CW  complex  is  finite  if  it  is  obtained   out  of  finitely  many  orbits  of  cells $S^n\times  G/H$ in  a  finite number  of  dimensions  $n$. 

This  is  equivalent  to  the fact  that  the  quotient $G/X$ is  compact (sometimes called  $G$-compactness or  co-compactness), which  is  also  equivalent  to    $X$  being  compact.  

Let  $G$  be  a discrete  group. Recall  that  given a  $G$-CW  complex  $X$,  the $G$-action on $X$  is  proper exactly  when  all  point stabilizers are  finite. This  is  equivalent  to  slice  conditions due  to  the  slice  theorem  \cite{palais}, \cite{antonyan}.

\begin{definition}\label{defgeomk}[Geometric Twisted  Equivariant  K-Homology]

Let $G$   be a  discrete  group and  let  $X$ be  a  proper  $G$-CW  complex,  let  $P$ be a projective, unitary stable $G$-equivariant bundle.

A geometric  cycle  is a  triple $(M,f,\sigma)$  consisting of 
 \begin{itemize}
 \item  A $G$-compact, proper  orientable, $\Spin(c)$  manifold  without  boundary $M$. 
 \item  A  $G$-equivariant map  $f: M\to X$. 

\item   A homotopy  class of  a  $G$-invariant  section  $\sigma:  M\to \Fred(-f^*(P))$  representing  an  element  on the twisted  equivariant  $K$-theory  group $$K_G^0(M, -f^*(P)).$$

\end{itemize}

\end{definition}

We  will  impose  some  conditions  on geometric  cycles in  order  to  obtain  geometric Twisted  Equivariant  $K$-Homology. 

\begin{enumerate}
\item Isomorphism. The cocycle  $(M,f, \sigma)$  is  isomorphic  to $(M^{'},f^{'}, \sigma^{'})$ if  there  exists  a   $G$-equivariant, $\Spin(c)$-structure preserving diffeomorphism $\psi: M\to  M^{'}$,  which  makes  the  following  diagram commutative: 

$$\xymatrix{M   \ar[rr]^-{\psi} \ar[rd]_-{f} & & M^{'}\ar[dl]^-{f^{'}}\\ & X &   
 }$$

\item Bordism. The  cycle   $(M_0, f_0,\sigma_0 )$is  bordant  to  $(M_1,f_1, ,   \sigma_1 )$  if  there  exists a  cycle  $(W, f,  \sigma )$,  where  $W$  is  a  $G$-compact, orientable  and $\Spin(c)$-manifold  with  boundary, $f: W\to X$ is  a  $G$-equivariant  map, and $\sigma:W\to \Fred(- f^*(P))$ is  the  homotopy  class  of  a  $G$-invariant  section,  with  the  property  that

$$(\partial  W, f\mid_{\partial W}, \sigma  )= (M_0,f_0,  \sigma_0 )\coprod (-M_1, f_1,  \sigma_1)  $$
 
 Where  $(-M_1, f_1, \sigma_1)$  has  the  reversed  orientation  with  respect  to $M_1$.  

\item Vector bundle modification. Given  a  geometric  cycle  $(M,   f,  \sigma )$, and an  even  dimensional real  vector  bundle with  a $G$-equivariant $\Spin(c)$-structure $F\to  M$,  denote  by $\pi:S(F\oplus  \IR)\to  M $  the    unit  sphere  bundle  of $F\oplus  \IR$. 
 We  consider  the map   $S_{M,F}:M\to  S(F\oplus  \IR)$ given  by the  $1$-section and  let  $S_{M,F}!:K_G^0(M, -f^*(P)) \longrightarrow  K_G^0(S(F\oplus  \IR), S_{M,F}^*(-f^*(P)))   $  be  the   pushforward  homomorphism.

Then,  the geometric  cycles  $(M, f,  \sigma )$  and  $S(F\oplus \IR),\pi\circ  f,   S_{M,F}!(\sigma))$ are  equivalent.  

We  will  introduce  the  notation $E\bigstar  (M,f, \sigma)$ for  the class  of  the  modification  of  the   geometric  cycle  $(M,f, \sigma)$  with  respect  to  the $\Spin(c)$-vector  bundle $E\to M$.   

Similarly,  we  will  denote  by $E\bigstar M$  the  manifold  given  as  the  total  space  of  the spherical  bundle $S(E\oplus \IR)$  over $M$.

 \end{enumerate}

The  following  result,  generalization of the  "Composition  Lemma"   on page  81 of  \cite{baumcareywang} in the  \emph{non-equivariant} setting,   will be useful to  study   interactions  between  bordism  and  vector  bundle modification: 
  
\begin{lemma}\label{lemmanormalbordism}
Let $M$  be  a proper, $G$-compact, smooth manifold, consider  a $G$-equivariant, $\Spin(c)$, smooth real vector  bundle of  even  dimensional  fibers, $\pi: E\to M$. Let  $s:M\to S(E\oplus  \IR)$  be  the  unit  section. 

Given  any  real, even  dimensional   smooth  vector  bundle $F$ with  a $\Spin(c)$-structure over  $S(E\oplus  \IR)$, 
the   proper $G$-manifolds 
$$F\bigstar (E\bigstar M), $$
and $$(s^*(F)\oplus E) \bigstar M  $$   
are bordant. 
\end{lemma}

\begin{proof}
Pick  up a $G$-invariant, hermitian metric  on $E\oplus  \IR$. This  is  possible because  the  action  on $M$  is  proper. Let $D(E\oplus \IR)$  be  the unit  Ball  bundle  therein.  Notice  that  the  boundary  of  $D(E\oplus \IR)$  is $E\bigstar M$. 

Consider  the  map $i: M \to  D(F\oplus  \IR)$  given  by  the  composition  of  the  1-section $M \to  S(E\oplus \IR)$  and  the  $0$-section  $S(E\oplus \IR)\to D(F\oplus \IR)$.

Let  $\nu$  be  the  normal  bundle  to  this  inclusion,  and  consider the  ball  bundle  inside  the  normal  bundle, denoted  by  $D_\epsilon( \nu)$  for  $\epsilon < \frac{1}{4}$. Then, the  sphere  bundle  of $s^*(F)\oplus E  \oplus  \IR$   is $G$-equivariantly  diffeomorphic  to the  boundary  of   the  Ball bundle  in $\nu$. 
Consider  the  manifold  with  boundary $W= D(E\oplus \IR)- D_\epsilon(\nu)$. This  is  a proper, $\Spin(c)$-manifold  with  boundary, which  is  a  bordism  between $F\bigstar (E\bigstar M)$  and  $(s^*(F)\oplus  E)\bigstar M$.

\end{proof}
\begin{corollary}

Given a   geometric  cycle  $(M, f, \sigma)$ and  real  orientable,   and  $\Spin(c)$-vector  bundles  $ \pi: E\to  M$  and $F\to  S(E\oplus   \IR) $,

Denote  by  $\rho_E:E\bigstar M \to M$ and  $\rho_F: F\bigstar (E\bigstar M)\to  E\bigstar  M$ the  projection  maps.  Let  $S_{M, E}: M\to E\bigstar M$ and $S_{M, S^*(F)\oplus  \IR }:M \to S^*F\oplus E \bigstar M $ be  the  $1$-sections into  the  spherical bundles.

 Then, the geometric cycles 
$$\big (  F\bigstar(E\bigstar M), f\circ \rho_E \circ \rho_F,  S_{M,E}!\circ S_{E\bigstar M, F}! (\sigma)\big )  $$
and  
$$ \big ( (S^*(F) \oplus E )\bigstar M, f\circ \rho_{S^*(F)\oplus  F},  S_{M, S^*(F)\oplus \IR} !(\sigma) \big )  $$ 
are  equivalent. 
\end{corollary}

Addition   on  the  set  of  geometric  cycles  is  defined  as  disjoint  union: 
$$( M, \sigma, f)+ (M^{'} \sigma^{'}, f^{'})=  ( M, \sigma, f)\coprod (M^{'} \sigma^{'}, f^{'}), $$
and  turns   the equivalence  classes  into  an  abelian  group.

\begin{definition}
Let  $X$  be  a  proper  $G$-CW  complex and  a projective  unitary stable  $G$-equivariant bundle  $P$ over  $X$. We  will denote   by  $\kgeom{G}{X}{P}$ the abelian  group  of geometric  cycles  on  $X$, divided  by  the  equivalence  relation  generated  by  Isomorphism,  Bordism  and  $\Spin(c)$-modification. 

 Denote  by  $\kgeom{G,0}{X}{P}$
 the  subgroup   generated  by  those  cycles $(M,f, \sigma )$ for  which  each  component of  $M$  has  even  dimension. 
 Similarly,  denote  by  $\kgeom{G,1}{X}{P}$   the  subgroup  generated  by  cycles  $(M,f, \sigma)$  for  which  each  component  of $M$ has  odd  dimension.    

\end{definition}

We  introduce  now  a  version  for  $G$-CW  pairs. 

Given  a  proper $G$-CW  pair $(X,A)$ and  a $G$-equivariant, projective  unitary stable  bundle  $P$ over  $X$, consider  the  triples $(M, \sigma, f)$  consisting  of  
 \begin{itemize}
 \item  A $G$-compact, proper  manifold   $M$ (possibly  with  boundary). 
 \item  A  $G$-equivariant map  $f: M\to X$  with $f(\partial M)\subset A)$.

\item   A homotopy  class of  a  $G$-invariant  section  $\sigma:  M\to \Fred(-f^*(P))$  representing  an  element  on the twisted  equivariant  $K$-theory  group $$K_G^0(M, -f^*(P)),$$

\end{itemize}
 
Define  on  the   set  of  geometric cycles   for  $(X,A)$ and  $P$  an  equivalence  relation  generated   by 
\begin{itemize}
\item Isomorphism. 
\item Bordism.
\item  $\Spin(c)$-vector  bundle  modification. 
\end{itemize}

 Denote  by  $\kgeom{G,*}{X,A}{P}$ the  set  of  such  equivalence  classes. 
 
As  before,  disjoint  union  defines   an abelian    group structure  on    $\kgeom{G,*}{X,A}{P}$. Notice that the  decomposition 
 
$$\kgeom{G,*}{X,A}{P}= \kgeom{G,0}{X,A}{P}\bigoplus \kgeom{G,1}{X,A}{P}$$
holds,  where the  summands $0$ and $1$  correspond  to  the  subgroup  generated  by  geometric  cycles  for  $(X,A)$ and  $P$   given  by  manifolds  of  even, respectively  odd dimensions. 

There  exists  a  boundary  map 

$$\delta: \kgeom{G,j}{X,A}{P}\longrightarrow  \kgeom{G,j-1}{A}{P},$$
 which is defined  as  follows. 

Given a  geometric  cycle  $(M, \sigma, f)$ for  $(X,A)$ and  $P$, the  cycle  $\delta ( M,\sigma, f)$   has  as  underlying  manifold $\partial M$.  The  map    $\partial M \to A$  is  given  as  the composition $\partial M\overset{j}{\to} M \overset{f}{\to} X$,  where $j$  denotes  the  inclusion  of  $\partial  M$ into $M$. The  section  $\partial M \to \Fred(- f^*(P))$ is  given  as  the  composition  of  the maps $j:\partial  M \overset{j}{\to} M \overset{f}{\to} \Fred(-f ^*(P))\overset{\mid j}{\to} \Fred (-f^*(P))$, where  the  map $\mid j$  restricts a  bundle  with  fiber $\Fred(-f^*(P))$ over  $M$ to a  bundle  over  $\partial M$. 
 
We  introduce an equivalence  relation  on   geometric  cycles,  which  is  more  technically  convenient. This  is  a  generalization  of  several  constructions  appearing  in the  literature \cite{emersonmeyertriviality},  \cite{baumcareywang} under  the  name  of  normal  bordism.

\begin{definition}\label{defnormalvectorbundle}
Let  $M$  be a  proper,  $G$- compact Manifold  together  with a  $G$-equivariant  map $f: M \to X$. Let $TM$  be  the  tangent  bundle  
A normal  bundle  for  $TM$   with respect  to  $f$  is  a  real, $G$-equivariant  vector  bundle  $ NM\to M$ such  that $TM\oplus NM$  is  isomorphic  to  a  pullback $f^*(F)$, where  $F$  is  a $G$-equivariant, $\Spin(c)$-vector  bundle  over $X$.  
\end{definition}

The  following  result  concerns  the  existence  of  normal  bundles: 

\begin{lemma}\label{lemmaexistencenormalbundles}
Let  $X$  be a  finite,  proper  $G$-CW  complex. Let  $M$  be  a  proper,  $G$-compact manifold with  a $G$-equivariant  map $f:M\to X$ . Then, there  exists  a  normal  bundle   for  $TM$  with respect to  $f
$.   
\end{lemma}

\begin{proof}

This  follows  from  Lemma 3.7 in  page  599  of  \cite{lueckoliverbundles}, after  noticing  that  the  manifold $M$ has  a  proper  $G$-CW structure, which  consists of  finitely  many  cells. 
   
\end{proof}

\begin{definition}\label{defgeneralizedbordism}[Normal  Bordism  for  cycles]
Two  cycles 
$$(M,f,  \sigma),$$ and  
$$(M^{'}, f^{'},  \sigma{'})$$   
for  $\kgeom{G,*}{X,A}{ P}$
are said  to  be  normally  bordant if  there  exist  finite  dimensional,  real orientable vector  bundles for  $TM$  with   respect  to  $f$,  denoted  by $E\to M$,  and  a  normal  bundle  for  $T M^{'}$  with  respect  to  $f^ {'}$ denoted  by   $E^{'}\to M^{'}$,

such  that  the  cycles

$$\big (E\bigstar M,\pi_{E}\circ  f,  S_{M, E}!(\sigma)\big )$$
 
and  
 
 $$ \big (E^{'}\bigstar M^{'},\pi_{E^{'}}\circ  f^{'},   S_{M^{'},E^{'}}!(\sigma^{'}) \big )  $$

are  bordant.

\end{definition}

\begin{lemma}
Two cycles  are  normal bordant  if  and  only  if  they  are  equivalent. 

\end{lemma}

\begin{proof}
The  implication from  normal  bordism  to  equivalence in the  sense  of definition \ref{defgeomk} is  clear. 

For  the converse, suppose  that  $(W,f,  \sigma)$  is  a  bordism   from  the  cycle $(M_0, f_0, \sigma_0)$  to  $(M_1, f_1,  \sigma_1)$.
 Choose a  normal  bundle $NW$  for  $TW$  with respect  to  $f$  in a  way  that $NW\mid_{\partial W}$ is a  normal  bundle  for   the  boundary. Then  $NW\bigstar W$  gives  a  normal  bordism   between  $\Spin(c)$-modifications  along  normal  bundles  for  the  boundary  components.

The  vector  bundle modification  procedure  also  implies normal   bordism. Let $E \to  M$  be  a real, $\Spin(c)$,  $G$-equivariant vector bundle.  Let  $F\to S(E\oplus \IR)$ be  a $\Spin(c)$, $G$-equivariant  real  vector  bundle, and  let  $N$  be  a  normal  bundle  with  respect  to the  map  $S(E\oplus  \IR)\to M \to X$ . 
 
Consider  the map  $s:M\to  F$  given as  the  composition   of  the $1$-section $s: M\to S(E\oplus  \IR)$,  and  the  $0$-section $S(E\oplus \IR)\to F$.

Now,  consider  the  bundle $s^*(N)\oplus  F \mid M $, and  notice  that  it  is a  normal  bundle   for  $M$ with  respect   to  the  map $f$.

 After  lemma \ref{lemmanormalbordism}, these total  spaces  are  bordant,  proper  $G$- $\Spin(c)$-manifolds.  The  bordism  determines a  bordism  of  geometric cycles.  
\end{proof}

We  now  prove  the  usual  homological  properties  of  twisted  equivariant  $K$-homology  for  proper  actions  of  a  discrete  group.

\begin{theorem}[Homological  Properties] \label{theoremhomologicproperties}
The   groups  $\kgeom{G,j}{X,A}{P}$ satisfy the  following  properties. 
\begin{enumerate}
\item (Functoriality). A $G$-equivariant map  of  pairs $\psi:(X,A)\to (Y,B)$ induces  a  group  homomorphism $ \psi^*:  \kgeom{G,j}{X,A}{\psi^* P} \to \kgeom{G,j}{Y,B}{P} $
\item (Homotopy  invariance). Two $G$-equivariantly  homotopic maps  
$\psi_0:(X,A)\to (Y,B)$  and  $\psi_1:(X,A)\to (Y,B)$ induce  the  same   group homomorphism $\psi:=\psi_0= \psi_1: \kgeom{G,j}{X,A}{P}\to \kgeom{G,j}{Y,B}{\psi^*(P)}$. 

\item (Six term  exact  sequence).  For  a  proper  $G$-CW  pair  $(X,A)$, the  following   sequence  is  exact: 

$$\xymatrix{ \kgeom{G,0}{X,A}{P} \ar[r]_-{\delta} & \kgeom{G,1}{A}{P} \ar[d]_-{i^ *} \\ \kgeom{G,0}{X}{P}  \ar[u]_-{j_*} & \kgeom{G,1}{X}{P} \ar[d]_-{j_*} \\ \kgeom{G,0}{A}{i^*P} \ar[u]_-{i_*} &  \kgeom{G,1}{X,A}{P} \ar[l]_-{\delta}  }, $$
where the  maps  $i_*$, $j_*$  are  induced  by  the  maps  of  proper  $G$-CW  pairs   $i: A\to X$ and  $j: X:= (X,\emptyset)\to (X,A)$  .   

\item  (Excision).  Given  a $G$-CW  pair $(X,A)$ and  a  $G$-invariant,  open  set $U$  with  the  property  that  the  closure of  $U$  is  contained  in  the  interior  of $A$,  the  inclusion $l:(X-U, A-U)\to (X,A)$  induces an  isomorphism  of  abelian  groups  
$$\kgeom{G,j}{X,A}{P} \to\kgeom{G,j}{X-U,A-U}{l^{*}(P)}   $$ 

\item(Disjoint  union   Axiom). Let $X=\underset{\alpha\in A} {\coprod} X_\alpha$  be  a proper  $G$-CW  complex  which  is  obtained  as  the  disjoint union  of  the  Proper $G$-CW  complexes  $X_\alpha$.  Then,  the inclusions $X_\alpha\to  X$ induce  an  isomorphism  of abelian  groups  

$$\bigoplus_{\alpha \in A }\kgeom{G,j}{X_\alpha}{P\mid X_\alpha}\cong  \kgeom{G,j}{X}{P} $$

\end{enumerate}
\end{theorem}
\begin{proof}
\begin{enumerate}
\item Associate  to  the  geometric  cycle $(M, \sigma, f)$ for  the  pair $(X,A)$  the cycle  $(M, \psi^*(\sigma) , \psi\circ f)$. 
     
\item It  follows from  the   bordism  relation. 

\item We  will  check  exactness  of  the  sequence  at $\kgeom{G,0}{X}{P}$.  Let $(M, f, \sigma )$  be  a   geometric  cycle  for  the  pair  $(X,A)$ which  maps  to  $0$  under  the  map $\delta$.  According  to  Lemma  \ref{lemmanormalbordism},  there  exists  a real,   orientable and  $\Spin(c)$ $G$-vector  bundle $E$ of  even dimension  over  $\partial M$  such  that the  geometric  cycle $E\bigstar (\partial M, f\mid_{\partial M}, \sigma\mid_{\partial M})$ is  twisted   null bordant.  We  can  assume  that  $E$  is  defined  on  all  of  $M$,  otherwise  using  lemma  3.7  of \cite{lueckoliverbundles}, page 599. 

Consider  a  geometric cycle $(W,F,\Sigma)$,  where  $W$  is  such that $\partial W = E \bigstar \partial M $, $F\mid_{\partial W}= \pi_{E}\circ  f$,   $ \Sigma \mid_{\partial W}= \sigma\mid_{\partial M}$. 

Form  the  manifold $  \hat{W}:=W\cup_{\partial W= \partial M } E\bigstar M $, where the  inclusion  of  $\partial  M$  into  the  second  factor  is  given as  composition  of  the  inclusion $\partial M \to  M$ together  with  the  unit  section $M\to  E\bigstar M$. 
The manifold  $\hat{W}$ together  with the  map $F\cup_{\partial M} f$,  and  the  class $\hat{\Sigma}= \Sigma \cup_{\partial M}  \sigma$  give  a geometric  cycle  in $\kgeom{G, *}{X}{P}$ which maps   
 to  the   given  cycle    for  $(X,A)$. 

Exactness  at  other  points  of  the  sequence  is  verified analogously.

\item We  will  construct  an inverse  to  the  map  induced  by  the  inclusion  $l$.
Recall the  extension  of Morse  Theory  to  proper  actions  \cite{Hingstonequivariantmorsetheory}, Theorems A and  B in page  94.  It  can  be  concluded  that  given a  geometric  cycle  $(M, f, \sigma) $ for  $(X,A)$,  there  exists a smooth $G$-invariant  function $\epsilon :M\to  \IR$ with  a   number  of  critical  orbits, which  are  all generic. 

We  will  assume  that the   function  $\epsilon$ separates  $f^{-1}(  \overline{U})$ and $M - f^{-1}(A)$. Then,  there    exist real numbers  $a<b$  such  that  $a= \sup_{x\in M-f^{-1}(A) }\epsilon (x)$,  and $b= \inf_{x\in f^{-1}(\overline{U})} \epsilon (x)$. 

For  a   regular  value  $c\in (a,b)$,   $\epsilon^{c}= \epsilon^{-1}(-\infty, c])$  can  be  furnished  with a  smooth  structure. 
The  boundary  component $\epsilon^{-1}(c)\cup \epsilon^{-1}(-\infty, c])\cap  \partial  M$ is  mapped  under $f$  to  $X-U$. 
  
Then, the  restriction of  the  cycle  to $M_1$  gives  a   geometric  cycle  $$(M_1, f\mid_{M_{1}},  \sigma\mid_{M_1})$$  for  the  pair $(X-U, A-U)$.  This  construction  gives  an  inverse  to  the map $l_*$.

\item This   follows easily  from  manipulations on  the  components   of  the manifolds involved  in the  geometric  cocycles.

\end{enumerate}

\end{proof}

Given  discrete  groups  $H$ and $G$, a  group  homomorphism $\alpha: H\to G$, and a proper   $H$-CW  pair $(X,A)$, the  induced space ${\rm ind}_{\alpha}X,$ is defined  to be  the  $G$-CW complex  defined  as  the  quotient space  obtained  from $G\times X $ by  the  right  $H$-action  given  by $(g,x)\cdot h =( g\alpha(h),h^{-1}x)$. The  following  result  summarizes  the  behaviour   of  the geometric  model  for  twisted  equivariant  $K$-homology   with  respect  to  induction    and  restriction procedures  on  spaces. 
Recall  the  existence  of  the   induction  homomorphisms  in Borel  cohomology  ${\rm ind}_{\alpha}:H^3( {\rm ind}_{\alpha} X\times_G EG, \IZ) \to   H^3(X\times_H EH, \IZ)$ and  in Twisted  Equivariant  $K$-Theory $ {\rm ind}_{\alpha}:  K_{G}^*({\rm ind}_{\alpha} X, {\rm  ind }_{\alpha }P) \to  K_{H}^*(X, P)$ \cite{barcenasespinozajoachimuribe}, (Lemma 5.4  in page 1347).

\begin{theorem}\label{lemmainduction}
There  exist  natural  group homomorphisms
$$ {\rm ind}_{\alpha}:K^{{\rm geo}}_{H, j}(X,A)\longrightarrow K^{{\rm geo}}_{G, j}({\rm ind}_{\alpha}(X,A), {\rm ind}_{\alpha }P)$$ 
associated  with a homomorphism $\alpha:H\to G$   which satisfy  the  following  conditions
\begin{enumerate}
\item{$\ind_{\alpha}$ is an  isomorphism whenever $\ker \alpha$ acts  freely  on $X$.}
 \item{For  any  $j$, $\partial_{n}^{G}\circ {\rm ind}_{\alpha}= {\rm ind}_{\alpha}\circ \partial_{n}^{H}$.}
 \item{For any  group  homomorphism $\beta: G\to K$ such  that $\ker \beta\circ \alpha$ acts freely on  $X$,  one  has 
$${\rm  ind}_{\alpha\circ \beta}= K^{{\rm geo}}_{K, j}(f_{1}\circ {\rm ind}_{\beta}\circ{\rm ind}_{\alpha}):K^{{\rm geo}}_{H, j}(X,A)\to K^{{\rm geo}}_{K, n}({\rm ind}_{\beta\circ \alpha}(X,A))$$
where $f_{1}: {\rm ind}_{\beta}{\rm ind}_{\alpha}\to {\rm ind}_{\beta\circ\alpha}$ is  the  canonical $G$-homeomorphism.}                                      

\item{For  any $j\in \mathbb{Z}$, any  $g\in G$, the  homomorphism 

$${\rm ind}_{c_(g):G\to G}K^{{\rm geo}}_{G, j}(X,A)\to K^{{\rm geo}}_{G, j}(\rm ind)_{c(g):G\to G}(X,A))$$

agrees  with  the  map  $\mathcal{H}_{n}^{G}(f_{2})$, where  $f_{2}: (X,A)\to {\rm ind}_{c(g):G\to G}(X,A)$ sends $x$  to  $(1,g^{-1}x)$ and $c(g)$  is  the  conjugation  isomorphism in $G$.}
\end{enumerate}

\end{theorem}

\begin{proof}
 
Given  a  geometric  cycle  for  $(M, f, \sigma) $  for $\kgeom{H}{X}{P}$,  form  the  cycle

 $$({\rm ind}_{\alpha}M,{\rm ind}_{\alpha}f, \Sigma ),$$
 where  $\Sigma: {\rm  ind}_{\alpha}X \to  \Fred(P)$  is  the  map determined  by  $(g,x) \mapsto g \sigma(x)$.   
  This  is  compatible  with the  equivalence  relations  defining  the  geometric  $K$-homology  groups.  The  homomorphism  is  seen to  satisfy  i-iv. 

\end{proof}

\subsection{The  spectral  sequence}

In order  to  proof  the  coincidence  of  geometric twisted equivariant  $K$-homology,  we  will  introduce  a  spectral  sequence  based  on  $G$-equivariantly  contractible  covers.  This  is  the  dual (homological instead  of  cohomological)  counterpart  of  the  spectral  sequence  constructed  in \cite{barcenasespinozauribevelasquez}  for   twisted  equivariant  $K$-theory. Due  to  the  poor  properties  of \v{C}ech   versions  of  homology (particularly  its  failure to satisfy exactness and  disjoint union  axiom  in  general),  we  will  make  the  following  assumption on the  space  $X$. 

\begin{condition}\label{conditionCMP}[Condition CMP]
The    $G$-space   $X$  is  said  to  satisfy  condition CMP  if $X$ is   a  Compact, Metrizable Proper  $G$-absolute  neighbourhood  retract (ANR).
\end{condition}  

These  conditions  are  certainly  met  in the  case  of  smooth,  proper  and  compact  $G$-manifolds with  isometric  actions, as  well as spaces  with a  proper  $G$-CW  structure with  a  finite  number  of  cells.

The  input  for the  spectral sequence  is  given  in terms  of  Bredon  homology  with  local  coefficients  associated  to  a contractible cover. We  will  introduce  some  notation relevant  to  this.

A detailed  description  of  the  homological  algebra  involved  in  Bredon  cohomology,  as  well  as  computations  related  to  twisted  equivariant  $K$-theory  and  Bredon  cohomology  can  be  found  in  \cite{barcenasvelasquez}, \cite{barcenasespinozauribevelasquez}.

Let  $\calu= \{ U_\sigma \mid  \sigma \in  I \}$  be  an   open  cover  of  the  proper $G$-space $X$ satisfying  hypothesis \ref{conditionCMP}. Assume  that each  open  set  $U_\sigma$   is   $G$-equivariantly homotopic  to  an  orbit $G/H_\sigma\subset U_\sigma$  for a  finite  subgroup  $H_\sigma$,  and   each   finite  intersection is  either  empty    or  $G$-equivariantly  homotopic  to  an  orbit. 

The  existence  of  such  a  cover,  sometimes  known  as  \emph  {contractible  slice  cover},   is guaranteed  for  proper  $G$-ANR's  by an appropriate  version  of  the  Slice  Theorem (see \cite{antonyan}). 

Let $\cover$ be  the category  where  the  objects are  non-empty  finite  intersections  of $\mathcal{U}$,  and  where  a  morphism   is  an  inclusion. 

\begin{definition}

A  covariant  system   with values  on  $R$-Modules   is  a  covariant    functor  $\cover\to  R-{\rm Mod}$. 
\end{definition}

\begin{definition}

  The nerve  associated  to  the  cover $\mathcal{U}$, denoted by  $\cover^*$  is  the  simplicial  set where  the  set  of  $0$-simplices $\cover^0$ is  given  by   elements  of  the  cover $\mathcal{U}$. 
   The  $1$ -simplices  in  $\cover^{1}$ are  non-empty    intersections $U_{\sigma_{1}}\cap U_{\sigma_{2}}$ an  in  general,  for  any  natural  number $n$,  the  $n$-simplices, $\cover^{n}$   are  non-empty  intersections  of  elements  in the  cover  $\mathcal{U}$  of  length $n+1$, $U_{\sigma_1}\cap \ldots \cap U_{\sigma_2}$. 

There are    face $\delta_i: \cover^{i}\to \cover^{i-1}$ and  degeneracy  maps $s_i: \cover^{i-1}\to \cover ^i$ coming  from  the associativity, respectively  idempotence of intersection.  

\end{definition}

\begin{definition}
Let  $X$  be  a  proper  $G$-space satisfying  condition \ref{conditionCMP}.   Given  a  contractible slice  cover  $\mathcal{U}$, and  a  covariant  coefficient  system $M$, the Bredon  equivariant  homology  groups  with  respect  to $\mathcal{U}$, denoted  by $H_{n}^G(X,\mathcal{U}; M),$ are  the  homology  groups  of  the  chain  complex defined  as  the  $R$-module 
$$ C_*( \cover, M) $$ 

Defined  as  follows. On  a  given  degree $n$,  $C_n(\cover, M)$  is  the  $R$-module  
$$\underset{U_{\sigma}= U_{\sigma_{1}} \cap \ldots \cap U_{\sigma_n} \in \cover^{n}}\bigoplus M(U_\sigma)\otimes_R  R[U_\sigma],  $$
where $R[U_\sigma]$ denotes  a   free $R$ module  of  rank  one and  $U_\sigma $  denotes a  non empty  length  $n$ intersection. 

The boundary  maps  of  the  simplicial  set $\cover$ tensored    with  the  $R$-module  homomorphisms  given  by  the   functoriality  of $M$   give   the  differential 

\begin{multline*}
d_n:  \underset{U_{\sigma}= U_{\sigma_{1}} \cap \ldots \cap U_{\sigma_{n+1}} \in \cover^{n}}\bigoplus M(U_\sigma)\otimes_R  R[U_\sigma] \\ \longrightarrow \underset{U_{\tau}= U_{\tau_{1}} \cap \ldots \cap U_{\tau_{n}} \in \cover^{n-1}}\bigoplus M(U_\tau)\otimes_R  R[U_\tau].
\end{multline*}

\end{definition}

We  will  analyze  now  the  coefficient  system  associated  to   twisted  equivariant  $K$-homology.

 Let $X$  be a  space  satisfying condition  \ref{conditionCMP}.  $i_\sigma: G/H_\sigma\to  U_\sigma \to X$ be  the inclusion of  a proper  $G$-orbit into  the  contractible  slice  $U_\sigma \subset X$ and  consider  the  Borel  cohomology  group  $H^3( EG\times _G G/H_\sigma, \IZ )$. Given a class $ P \in  H^3(EG \times_G X, \IZ)$,  we  will  denote  by  $\widetilde{H_{P_\sigma}}$  the  central  extension  $1\to S^1 \to \widetilde{H_{P_\sigma}} \to  H_\sigma \to  1 $
associated  to  the    class  given  by  the  image of  $P$  under  the  maps 
$$\omega_\sigma : H^3(EG\times X, \IZ)\overset{i_\sigma ^* }{\to} H^3(EG\times_G G/H_\sigma, \IZ )\overset{\cong}{\to} H^3(BH_\sigma, \IZ)\overset{\cong}{\to} H^2(BH_\sigma , S^1).  $$

Denote  by  $R(\widetilde{H_{P_\sigma}})$  the  abelian  group  of isomorphism  classes  of  complex  representations  of  the   group $\widetilde{H_{P_\sigma}}$. Let $R_{S^1}(\widetilde{H_{P_\sigma}})$  be  the  subgroup  generated  by the complex  representations   of  $\widetilde{H_{P_\sigma}}$ on  which  the central subgroup $S^1$  acts  by  complex  multiplication.

\begin{lemma} [Coefficients  on  equivariant contractible  covers]

The  restriction  of the   functors  $\kgeom{G,0}{X}{P}$ and $\kgeom{G,1}{X}{P}$ to  the  subsets $U_\sigma$ gives covariant  functors  defined  on the  category $\cover$. 
 As  abelian  groups, the  functors  $\kgeom{G,j}{X}{P}$  satisfy: 

$$\kgeom{G,j}{U_\sigma}P\mid_{U_{\sigma}} = \begin{cases} R_{S^{1}}(\widetilde{H_{P_\sigma}})\;  \text{If}\;  j= 0   \\  0 \, \text{If}\;  j= 1    \end{cases}.$$

In  degree  $0$, a   map $ \kgeom{G,0}{U_\sigma}P\mid_{U_{\sigma}} \to \kgeom{G,0}{U_\tau}P\mid_{U_{\sigma}}$  is  given  by  induction  of  representations along  group   inclusions  $\widetilde{H}_\sigma\to  \widetilde{H}_\tau$. 

  \end{lemma}

\begin{proof}
Notice  the  group  isomorphism $\kgeom{G,j}{U_\sigma}P\mid_{U_{\sigma}}\cong \kgeom{G,j}{G/H_\sigma}P\mid_{U_{\sigma}}$ due to homotopy  invariance, part ii in Theorem \ref{theoremhomologicproperties}. The  induction structure    gives  an  isomorphism 
$$\kgeom{G,j}{G/H_\sigma}{P\mid_{U_{\sigma}}} \cong  \kgeom{H,j}{\pt }{\omega_\sigma (P)}.$$
Consider the   $H_\sigma$-trivial  Space $\pt$, and   a  geometric cycle  $$\big (M, f:M\to \pt, s: M\to \Fred(f^*(\omega_{P_\sigma})) \big ), $$
representing  an  element  in the  geometric  equivariant  $K$-homology  group 
 $$\kgeom{H,j}{\pt }{\omega_\sigma (P)}.$$ 
The  pushforward map $f_!(s): K_{H_\sigma}^j(M, f^*(\omega_{P_\sigma}))\to K_{H_\sigma}^j(\pt, \omega_{P_\sigma}) $ defined  in \cite{barcenascarrillovelasquez} associated  to  the  map  $f: M\to  \pt$ gives  an  isomorphism  to  the group $K_{H_\sigma}^j(\pt, \omega_{P_\sigma})$. On the other  hand,  this  group  has  been  verified  to  be  isomorphic  to  either 0  if  $j$  is  odd  or   the group $R_{S^{1}}(\widetilde{H_{P_\sigma}})$  if  $j$  is  even  in  \cite{barcenasespinozajoachimuribe}, 5.3.4.     

For  the   description of  the  induced  morphism,  notice  the  compatibility  of  the  induction  structure  in the  sense  of  \ref{lemmainduction}  with 	  the  representation theoretical  induction  structure  of  $ R_{S^{1}}(\widetilde{H_{P_\sigma}})$.  This  follows from  the  naturality  of  the equivariant  topological index   in the  context  of  finite  groups \cite{atiyahsegalindex}. 

\end{proof}
We  will  now  give  a  description  of Segal's spectral  sequence \cite{Segal},  used  to compute Twisted Equivariant Geometric  $K$-homology.  The  main  application  will be   Theorem \ref{theoremcoincidencekhomology} below, relating  these  groups  with analytical  approaches  to  twisted  equivariant  $K$-homology. 
The  construction  of  the  spectral  sequence is  completely  analogous  to that  of \cite{barcenasespinozauribevelasquez}. 

\begin{proposition} \label{proposition spectral sequence  twisted}
Let  $X$  be  a   $G$-space  satisfying  condition \ref{conditionCMP}. There  exists  a spectral sequence
associated to a locally finite and equivariantly contractible cover $\mathcal{U}$.  It converges to  $\kgeom{G}{X}{P}$, has for second page $E^2_{p,q}$ the homology of $\cover$
with coefficients in the functor $\mathcal{K}_0^G( ?,P|_?)$ whenever $q$ is even, i.e.
\begin{equation}
E^{2}_{p,q}:= 
H_{p}^G(X,\mathcal{U}; \mathcal{K}_{\rm geo, 0}^ G( \calu ,P|_{\calu})) 
\end{equation}
and is trivial if $q$ is odd.
Its higher differentials
$$d_{r}:E^{r}_{p, q}\to  E^{r+1}_{p-r, q+r-1}$$
vanish  for    $r$ even.

\end{proposition}

\begin{proof}
The cover consists of  open  spaces,  which  are $G$-equivariantly homotopic  to  an  orbit.  Hence,   we know that the groups $K_{q, {\rm geom}}^G(U_\sigma; P|_{U_\sigma})$ are periodic and trivial for $q$ odd. Hence, the  $E^2$  term   is  isomorphic to $H_{p}^G(X,\mathcal{U}; \mathcal{K}^q_G( ?,P|_{?}))$. 

\end{proof}

\section{Analytic  $K$-homology and  the  Index  Map.}\label{analyticalkhom}
Twisted  Versions  of  Equivariant  $K$-theory have  been   considered  in the  literature  using   a number of  analytical constructions  involving  $KK$ theory \cite{echterhoffchabertgroupextensions}, \cite{tuxustacks}.

We  will  use  the  coincidence    of  these  approaches, previously  studied  in the  literature \cite{tuxustacks},  \cite{emersonmeyer},   as  well as  duality  results  to define an  analytical index  map with  image   on  an    equivariant  $KK$-theory  group,  which is  verified  to  be  an  isomorphism  for proper  finite $G$-CW  complexes.

We  shall recall  notations  for  equivariant  $KK$-groups. References  for  further study  include  \cite{blackadar}, \cite{mislinvalette}.

 \begin{definition}
Let  $A$ and  $B$  be  graded  $G$-$C^*$  algebras. The  set $\mathcal{E}_G(A,B)$  of  equivariant  Kasparov  modules  consists of  triples  $(E,\phi, T)$  of  a  countably  generated, graded  Hilbert $B$-module $E$  with a  continuous  action  of $G$, an  equivariant  graded  $*$-homomorphism  to the  adjointable  operators  of  the  Hilbert  Module  $E$  $\phi: A\to \mathbb{B}(E)$, and  a  $G$-continuous  operator $F$  of  degree  $1$  in  $\mathbb{B}(E)$, with  the  following  properties: 
\begin{itemize}
\item The  commutators $[F, \phi(a)]$  are  in  the  algebra  of  compact, adjoinable  operators $ \mathbb{K}(E)$ for  all $a\in A$.  
\item    $(F^2-1) \phi(a) \in \mathbb{K}(E)$.  
\item  $ (F-F^*)\phi(a)\in \mathbb{K}(E)$. 
\item  $(gF-F)\phi(a) \in \mathbb{K}(E)$. 
\end{itemize}

Here  $\mathbb{K}(E)$ denotes  the  compact  adjoinable  operators  of  the  Hilbert  Module  $E$. 
Direct  sum   and  homotopy  turn  $KK_G$  into  a  group,  \cite{blackadar} 20.2  in page  206.  Usings  Kasparov's   Bott  periodicity, \cite{blackadar} 20.2. 5, these groups  are  extended  to  $\mathbb{Z}/2$-graded  groups $KK_G^*( \quad, \quad)$,  which  are  covariant  in  the  second variable,  respectively  contravariant  in   the  first  one.  

\end{definition}

\begin{definition}
Let $X$  be  a  proper  $G$-CW  complex. Given a  projective  unitary  stable $ G$-equivariant   bundle  $P$  defined   with  respect  to  a   stable  Hilbert space $\HH$,  consider  the  Space  $\KK$ of  compact operators  in  $\HH$. Let  the  group $P\UU( \HH)$  (with  the $*$-strong  topology), act  by  conjugation  on the   space  $\KK$  of  compact  operators  on $\HH$ with the  norm  topology (which  agrees  with  the $*$-topology.

The continuous  trace  Bundle  associated  to $P$  is  the  bundle with  structural group  $P\UU( \HH)$ and fiber $\KK$ associated  to  $P$,  denoted  by $\KK_P=  P\times_{P\UU(\HH)} \KK$.   

\end{definition}
The  following  result  summarizes  some  properties of  the  continuous  trace  bundle  associated  to  a  twist $P$. 

\begin{proposition}
Let  $P$  be  a  projective  unitary  stable $ G$-equivariant   bundle  $P$ defined  over  a  $G$-CW  complex  $X$. Then,  
\begin{enumerate}
\item The  bundle  $\KK_P$  is  locally  trivial. 
   
\item The  sections vanishing  at  infinity  of $\KK_P$  form a  $G$-Hilbert  Module over   the  $C^*$-algebra  $C_0(X)$, denoted  by  $A_P$.

\end{enumerate}
\end{proposition}

\begin{proof}
We  will verify the   statements.
\begin{enumerate}
\item It  follows  from  the  local  triviality of  the  projective  unitary  stable  $G$-Bundle  $P$. 
\item  This  follows  from  local  triviality.

\end{enumerate}
\end{proof}

\begin{theorem}\label{theoremequivalencekk}
Let  $X$  be a  proper, $G$- compact .  There  exist group  isomorphisms 
$$A: \ktheory{G}{*}{X} \longrightarrow  \mathcal{R} KK_{G, X}^*(C_0(X), A_P) $$

$$B: \mathcal{R} KK_{G, X}^*(C_0(X), A_P) \longrightarrow  KK( \IC, A_P\rtimes G) $$ 
Where  the  right  hand  side of  the  first  equation,  respectively the  right hand  side  of  the  second  one   is   representable  $KK$-theory \cite{kasparovrepresentable}  with  respect  to  $C_0(X)$.
\end{theorem}

\begin{proof}

\begin{itemize}
\item This  is a  consequence  of Theorem  3.14  in page 872  and  proposition 6.10 in page  900  of  \cite{tuxustacks}.  However, we  will  collect  the  needed  results. First  notice  that    the  action  groupoid  $\Gamma:= X \rtimes G$  is  proper  and  locally  compact. The continuous  trace  Algebra  bundle $A_P$  carries  the  structure  of  a  Fell  bundle  over $\Gamma$. Tu-Xu-Laurent  construct in \cite{tuxustacks}  the  reduced  $C^*$ algebra  associated  to  a  Fell  bundle   and  denote it  by  $C_r^*(\Gamma, E)$.     

 Specializing Proposition  6.11 in page  900  of  \cite{tuxustacks}  identifies the  Kasparov KK-group
 $$KK_{\Gamma}^*(C_0(X), C_0(X, \KK_P))$$  
 with the  $K$-Theory  groups  $K_*(C_r^*(\Gamma, \KK_P)).$
 Theorem 3.14   in \cite{tuxustacks}  identifies  this  $K$-Theory  group  with the  twisted  equivariant  $K$-theory  groups  $\ktheory{G}{*}{X} $ as  introduced  in  \cite{barcenasespinozajoachimuribe}  in  terms  of homotopy  groups  of  spaces  of  sections  of  Fredholm  operators. 
 On the  other hand,  the   groups $KK_{\Gamma}^*(C_0(X), C_0(X, \KK_P)$ are  isomorphic  to  the   groups 
 $\mathcal{R} KK_{G, X}^*(C_0(M), A_P)$ of  representable,  equivariant  $KK$- theory  with  respect  to $X$.  This  gives  isomorphism  A.

 \item  For  the isomomorphism  $B$, consider the  composition 
 $$  \mathcal{R} KK_{G, X}^*(C_0(M), A_P) \longrightarrow KK (C_0(X) \rtimes G, A_{P}\rtimes G) \to KK ( \IC, A_P \rtimes G)  , $$ 
 where   the  first  map  is  Kasparov's  descent  homomorphism  and  the  second   is  given  by  tensoring  with  the  Mishchenko-line  bundle  for  $X$. It  is  proved  in \cite{echterhoffemersonkim}, Theorem 2.6  in page  15,  that this  map  is  an  isomorphism. 
\end{itemize}

\end{proof}

General  instances  of  the  following  theorem  have  been  proved  in \cite{echterhoffemersonkim} ( Theorem  1.2 in page  3), as  well as  \cite{emersonmeyer}, Section  7.3.  

\begin{theorem}\label{teoanalyticalpoincaredualityechtehoff}
Let  $G$ be  a  discrete Group. Let  $M$  be  a proper,  cocompact $\Spin(c)$ smooth $G$-manifold.  Let  $P \in H^3(X\times_G, EG,  \IZ)$  be  a   class  representing  a  projective  unitary stable $G$-equivariant  bundle. Denote  by $A_P$  the   continuous  trace algebra  associated  to  it  and  let $A_{-P}$  be the  continuous  trace  algebra  asssociated  to  the  additive  inverse of  $P$.

There  exists an  isomorphism  of  Kasparov  $KK$-groups
$$PD: \mathcal{R} KK_{G, X}^*( A_{P},  \IC) \longrightarrow KK_G^*(\IC,  A_{-P})$$

\end{theorem}

\begin{proof}
The  proof  consists  of two   $KK$-equivalences. 

$$  \mathcal{R} KK_{G, X}^*( A_{P}, \IC ) \cong \mathcal{R} KK_{G, X}^*(\IC, A_{-P}) \cong KK_G^*(\IC,  A_{-P})$$ 
The  first  isomorphism  comes  from Theorem  1.2  in  \cite{echterhoffemersonkim}. 
   The  last  equivalence  follows   from  the  compacity  assumption on  $M$.  
 
\end{proof}

\begin{theorem}\label{theoremcoincidencekhomology}
Let  $X$  be  a finite  proper  $G$-CW  complex. There  exists   an  isomorphism 
$${\rm Index}: \kgeom{G, *}{X}{P} \longrightarrow KK_G^*(\mathbb{C}, A_{P})$$
\end{theorem}

\begin{proof}
Consider  a  geometric  cycle  $(M,f,\sigma)$,  and  recall  that  it  consists  of  a $G$-compact, proper  orientable, $\Spin(c)$  manifold  without  boundary $M$,  a $G$-equivariant map  $f: M\to X$,  as  well  as  homotopy  class of  a  $G$-invariant  section  $\sigma:  M\to \mathcal{K}_{-f^*(P)}$  representing  an  element  on the twisted  equivariant  $K$-theory  group $K_G^0(M, -f^*(P))$.   Consider  the   Poincar\'e  duality  isomorphism \ref{teoanalyticalpoincaredualityechtehoff} and  the  equivalence  to  the  $KK$-approach  \ref{theoremequivalencekk}. They  give   isomorphims  fitting  in the  following  commutative  diagram: 

$$ K_{G}^*( X, -f^*(P)) \overset{A}{\cong} \mathcal{R}_{X}KK_G^*(A_{-f^*(P)}, \IC) \overset{PD}{\cong} KK_G^*(\IC, A_{f^*(P)})$$

Define  the  map  ${\rm Index}$  on the  cycle  $(M, f,  \sigma)$   as  the  element $ f_* \circ  PD \circ A ( \sigma)$,  where  $f_*$  denotes   the  map  induced  by  $f$ on  equivariant $KK$-theory groups.

Since  Both  functors  satisfy  the  axioms  of  an  equivariant  homology  theory, there  exist   spectral  sequences $E_{p,q}^r$  converging  to  
$\kgeom{*}{X}{P}$ , as  indicated  in \ref{proposition spectral sequence  twisted}. 	Using  the  usual  homological  properties  of  $KK$-theory,  one  constructs  out  of  the  cover  for  the  spectral sequence  in \ref{proposition spectral sequence  twisted},  a spectral  sequence   $F^r_{p,q}$  converging  to  $KK_G^*(\IC, A_{P})$.  

Moreover,  the ${\rm  Index} $ map   is  natural  and  gives  a natural  transformation  between  the   second  terms  of  the  spectral  sequences,    which  consists  of  isomorphisms for a  compact=finite=cocompact $G$-CW complex  $X$.
 \end{proof}

\section{Relation  to  twisted  crossed  products} \label{Section Crossed  Products}
We  state  briefly  the  relation  to  twisted crossed  products,   studied  previously  in   \cite{echterhoffwilliams}, \cite{packerraeburn}.  

Let  $G$  be  a discrete, countable  group.

\begin{definition}
 A  model  for  the  classifying  space  for  proper  actions  $\eub{G}$  is a  proper  $G$-CW   complex  $\eub{G}$  with  the  property  that,  given  any  other   proper  $G$-CW  complex  $X$,   there exists up  to  $G$-equivariant  homotopy a  unique  map $X\to \eub{G}$.  
\end{definition}

Models  for  $\eub{G}$ always  exists   and  are  unique up to  $G$-equivariant  homotopy \cite{lueckclassifying}, Theorem  1.9 in page 275. Moreover, $\eub{G}$  can  be  characterized  as a  $G$-CW Complex   with  isotropy  in  the  family  of  finite  subgroups  and  such  that   the  fixed  point  set $X^H$  is  either  contractible  if $H$  is  finite  or  empty  otherwise. 

Notice  in particular  that $\eub{G}$  is  contractible  after  forgetting  the  group  action. 

Hence, the  Borel  construction $ \eub{G} \times _G EG$ is   a  model  for  $BG$, the  base space  of  the  universal   principal  $G$-bundle $G\to EG \to BG$. This  gives  an  isomorphism  

$$\omega: H^3(\eub{G}\times_G EG, \IZ)\to   H^3(G, \IZ),$$
 
 where $H^3(G, \IZ)$  denotes group  cohomology with  constant  coefficients,  which  is  isomorphic  to  Borel-Moore  cohomology \cite{moorecohomology} since  $G$  is  discrete.  
 
Let $S^1$  be  the abelian   group of  complex  numbers  of  norm  one. 

There  exists an  isomorphism $\omega: H^3(G, \IZ)\to H^2( G, S^1)$,   which  is  given as  a  connection  homomorphism  of  the  long  exact  sequences produced  by  the  coefficient  sequence $1\to \IZ\to \IR \to S ^1\to 1 $.

 Given a  cohomology class $P \in  H^3(\eub{G}\times_{G} EG, \IZ)$,  we  will  denote  by $\alpha_\omega(P)$  a cocycle  representative   $\alpha_{\omega(P)}:  G\times  G \to  S^1$ of  the  class  associated  to  $\omega(P) \in H^2(G, S^1)=  Z^2(G, S^1)/ B^2(G,S^1)$. 
   
Given  such  a  cocycle  $\alpha_{\omega(P)}$,  one  can  form  the  reduced  and   full $\alpha_{\omega(P)}$-twisted  group $C^*$-algebras, as  follows. 

\begin{definition}
 Let $G$  be  a  discrete  group. Let $\alpha: G\times  G\to  S ^1$  be a $2$-cocycle. The  $\alpha$-twisted  convolution  algebra $l^1(G, \alpha)$  is  the  space  of  all  summable  complex  functions  on  $F$  with  convolution  and  involution  given  by 
 $$ f*_{\alpha }g( s)=  \sum_{t\in G} f(t)g(t^{-1}s)\alpha(t, t^{-1}s),  $$ 
 
 $$ f^*(s) =  \alpha(s, s^{-1})f( s^{-1}). $$ 
The  full  twisted  group algebra  $C^*(G,  \alpha) $  is  defined  as  the  enveloping  $C^*$  algebra of  $l^1(G, \alpha)$. 

The  reduced  twisted  group $C^*$-algebra  $C_{r}^*(G, \alpha)$  is  the  closure  of  the \emph{ $\alpha$-regular  representation} $L_\alpha: G\to U( l^2(G))$  given  by  
$$ (L_{\alpha}(s)\xi)(t)= \alpha(s, s^{-1}t) \xi(s^{-1}t), \; \xi \in l^2(G),  s,\,  t\in G $$

 \end{definition}
In a  similar  fashion,  given  a   discrete,  countable  group $G$,  a  cocycle  $\alpha:G\times G\to S^1$,  a  $C^*$ algebra  $A$   and  a  map  $U:  G \to Aut(A)$,
satisfying  $U(g_1 g_2)= \alpha(g_1, g_2) U(g_1) U_(g_2)$,  one    forms   
the  twisted    reduced  and full  crossed  product, denoted  by $A\rtimes_{\alpha}^r G$ respectively $A\rtimes_{\alpha} G$. See  \cite{packerraeburn}  for  details.

The  following  result  relates the Continuous  trace  Algebra $A_P$  with  a  twisted   crossed  product  construction. 

\begin{proposition}\label{propositioncontinuostraceeub}
The  $C_0(\eub{G})$  algebras  $A_P$  and  $K\rtimes^r_{\alpha_{\omega(P)}} G \otimes C_0(\eub{G})$  are  Morita  equivalent. 

\end{proposition}
 
\begin{proof}
Consider  the  action  groupoid  $\Gamma:= \eub{G}\rtimes  G$.  Both  $C_{0}(\eub{G})$-algebras  are  section  algebras  of elementary $C^*$ bundles   over  the  space  of  objects  $\eub{G}$  of  $\Gamma$. 

Theorem 16  in  page  229  of  \cite{kumjianelementary} identifies  the Morita  equivalence  of  such  bundles  in terms  of   classes in a cohomology  group  isomorphic  to  the Borel  construction $H^2(\eub{G}\times_G EG, S^1)$. For  the   bundles  in  question,  the  classes  are  representatives  of  the class of  $\omega(P)$,  which  is  the  same  as  the  one  of  $\alpha_{\omega(P)}$. 
See  also  proposition 1.5  in \cite{echterhofflueckphillipswalters}
\end{proof}

We  will  assume  now  that  the  group  $G$  has   a   model  for  $\eub{G}$  which   is  constructed  out  of  finitely  many equivariant  cells.  Some  examples  of  such  groups  include fundamental  groups  of  hyperbolic  manifolds (where the  universal  cover  is  a  model), Gromov  Hyperbolic  groups (where  the  Rips  complex  with specific  parameters  provides  a  model),   Mapping  class  groups of  surfaces (the  Teichm\"uller  space  gives  a  model  for  $\eub{G}$),  and   groups  acting  properly  and  cocompacly  on  $CAT(0)$-spaces.

\begin{theorem}
Let  $G$  be  a   discrete  group with  a  model  for  $\eub{G}$  with  finitely  many cells. Then,  there  exists a  commutative  diagram 

$$\xymatrix@u{  \mathcal{R} KK_{G, \eub{G}}^* (C_{0}(\eub{G}), A_{P}) \ar[r]^-{C} & \mathcal{R} KK_{\eub{G}}^*( C_0(\eub{G}) , C_0(\eub{G}) \otimes  \mathcal{K}\rtimes_{\alpha_{\omega(P)}}^r G) \ar[r]^-{BC}  &   K_{*}( \mathcal{K}\rtimes_{\alpha_{\omega(P)}}^r  G ) \ar[d]^{PR} \\ \kgeom{G, *}{\eub{G}}{P} \ar[u]^-{B^{-1}\circ \rm Index}   &   & K_*(C_r^*(G, \alpha_{\omega(P)}) ).  &  } $$

\end{theorem}
\begin{proof}
The map $PR$  is  an  isomorphism  due  to  the  Packer-Raeburn Stabilization   Trick,  Theorem 3.4 in  page  299 of  \cite{packerraeburn}. 

The  map $BC$,  the  Baum-Connes  Assembly  map  with  coefficients  in  $\KK\rtimes_{\alpha_{\omega(P)}}G$  is an isomorphism  because  of Theorem 5.4,  page 180   of \cite{kasparovskandalis}.   

The  map $C$  is  an  isomorphism because  the  $C_0( \eub{G})$-Morita  equivalence  of   \ref{propositioncontinuostraceeub} induces  an  isomorphism  in the  representable  $KK$-theory  group. 

The last  map  is  the  ${\rm Index}$ isomorphism  map  of  \ref{theoremcoincidencekhomology} composed  with  the  isomorphism $B^{-1}$.

\end{proof}

\bibliographystyle{abbrv}
\bibliography{GEOMETRIC}

\end{document}